\documentclass[12pt]{amsart}

\usepackage{latexsym,amsmath,amsfonts,amscd,amsthm,upref,graphics,amssymb}

\usepackage{mathrsfs}
\usepackage[all]{xy}
\usepackage{amsmath}
\usepackage{amssymb}
\usepackage{latexsym}
\usepackage{enumerate}
\usepackage{epsfig}
\usepackage{bar}
\usepackage{epic}
\usepackage{multicol}
\usepackage{graphicx}
\usepackage{epsfig}

\usepackage{graphicx}
\usepackage{wrapfig}
\usepackage{color}
\usepackage{graphicx}
\usepackage{epsfig}

\usepackage{ragged2e}
\usepackage{lipsum}
\usepackage{epstopdf}

\usepackage{epstopdf}

\usepackage{graphicx}
\usepackage{wrapfig}
\usepackage{color}
\usepackage{graphicx}
\usepackage{epsfig}

\usepackage{ragged2e}
\usepackage{lipsum}
\usepackage{epstopdf}

\usepackage{epstopdf}

\theoremstyle{plain}
\theoremstyle{definition}

\newtheorem{theorem}{Theorem}[section] 

\newtheorem{proposition}[theorem]{Proposition}
\newtheorem{remark}[theorem]{Remark}

\begin{document}
  \title[Nonsingular transformations and Dimension spaces ]{Nonsingular transformations and Dimension spaces }
  \author{Thierry Giordano}
  \address{T. Giordano, Department of Mathematics, University of Ottawa, 585 King Edward
  Ottawa, Canada K1N 6N5  }\email{giordano@uottawa.ca}
  \author{David Handelman}
    \address{D. Handelman, Department of Mathematics, University of Ottawa, 585 King Edward
    Ottawa, Canada K1N 6N5  }\email{dehsg@uottawa.ca}
  \author{Radu B. Munteanu}
  \address{R.-B. Munteanu, Department of Mathematics, University of Bucharest, 14 Academiei St.,  Sector 1, Bucharest,
  Romania} \email{radu-bogdan.munteanu@g.unibuc.ro}
  \thanks{The first and second author were partly supported by NSERC(Canada) operating grants and the third author was partly supported by by Grant PN-II-RU-PD-2012-3-0533 of the Romanian National Authority for Scientific Research, CNCS - UEFISCDI}

\maketitle
\begin{abstract}
For any adic transformation $T$ defined on the path space $X$ of an ordered Bratteli diagram, endowed with a Markov measure $\mu$, we construct an explicit dimension space (which corresponds to a matrix values random walk on $\mathbb{Z}$) whose Poisson boundary can be identified as a $\mathbb{Z}$-space with the dynamical system $(X,\mu,T)$. We give a couple of examples to show how dimension spaces can be used in the  study of nonsingular transformations.
\end{abstract}
\section{Introduction}

 The aim of this paper is to show how dimension spaces can be used in the study of nonsingular transformations.
 Let $(X,\mu)$ be a measure space and let $T$ be a nonsingular automorphism of $(X,\mu)$. Clearly, $x\mapsto T^nx, x\in X, n\in\mathbb{Z}$ is an action of $\mathbb{Z}$ on $X$, and so $(X,\mu)$ is a $\mathbb{Z}$-space. Throughout this paper a triple of the form  $(X,\mu,T)$ will be called a
 dynamical system.

 In [GH], Giordano and Handelman constructed an invariant for classification of ergodic dynamical systems  called dimension space. Since any dynamical system $(X,\mu, T)$ is the Poisson boundary of a certain matrix values random walk on $\mathbb{Z}$ (see \cite{EG}) and the Poisson boundary of a matrix values random walk on $\mathbb{Z}$ can be identified with the Poisson boundary of a $\mathbb{Z}$-dimension space (see \cite{GH}), any dynamical system can be regarded as the Poisson boundary of a certain $\mathbb{Z}$-dimension space. Even though to any dynamical system $(X,\mu,T)$ it corresponds a $\mathbb{Z}$-dimension space, no explicit construction was previously known. In this paper we show that if $T$ is an ergodic nonsingular transformation on a measure space $(X,\mu)$ which admits an adic realization (that is, $T$ is measure theoretically conjugate to an adic transformation on an ordered Bratteli diagram endowed with a Markov measure) then to the dynamical system $(X,\mu,T)$  we can associate a $\mathbb{Z}$-dimension space $H$, in an explicit way, such that $(X,\mu,T)$ is isomorphic to the Poisson boundary of $H$ (by this we mean that the two $\mathbb{Z}$-actions are measure theoretically conjugate)

 Let us briefly recall the definition of a dimension space. Throughout this paper we consider only $\mathbb{Z}$-dimension spaces. Let $\mathbb{Z}$ be the group of integers and let $A = \mathbb{RZ}$ be the group algebra which we view as a subalgebra of $l^{1}(\mathbb{Z})$ with respect to the counting measure. We identify $A$ with the Laurent polynomial algebra in one real variable $A = R[x^{\pm 1}]$. For a Laurent polynomial $f\in A$, the coefficient of $x^{k}$, will be denoted by $(f, x^{k})$. There is a natural ordering on $A$, with positive cone
  \[A^{+}=\left\{\sum r_{n}x^{n} | \ r_{n}=0 \text{ a.e., }r_{n}\geq 0\right\}.\]
  Let $k(n)$ be positive integers and let $A^{k(n)}$ be the set of columns of size $k(n)$. Each $A^{k(n)}$ is a partially ordered vector space with the direct sum ordering. Let $(M_{n})_{n\geq 1}$ be
  a sequence of positive $A$-module maps $M_{n} : A^{k(n)}\rightarrow A^{k(n+1)}$; in other words, $M_{n}$ is a $k(n + 1)\times k(n)$ matrix with entries $\varphi^{n}_{i,j}$
  from $A^{+}$. Let $H$ be the direct limit
  \[H=\underset{n\rightarrow \infty}\lim A^{k(1)}\overset{M_{1}}\longrightarrow A^{k(2)}\overset{M_{2}}\longrightarrow A^{k(3)}\overset{M_{3}}\longrightarrow\cdots .\]
  The limit $H$ consists of the set of pairs $\left\{(f,n); n\in \mathbb{N}, b\in A^{k(n)}\right\},$ modulo the equivalence relation
  \[(f, n)\sim (g,m)\text{ if there exists }l>n, m\text{ such that } M_{l-1}M_{l-2}\cdots M_{n}f= M_{l=1}M_{l-2}\cdots M_{m}g.\] The equivalence class of a pair $(f, n)$ with $n\in\mathbb{N}$ and $f\in A^{k(n)}$ will be denoted by $[f,n]$.

  The set $H^{+}$ of equivalence classes that contain a positive element is the positive cone for $H$. Now $\mathbb{Z}$ acts on $A$ via multiplication by $x$, and obviously this action extends to an action on each $A^{k(n)}$, and thereby to the limit space, $H$, via $[f, n] \mapsto[f\cdot x, n]$. Notice that matrix multiplication acts on the left, commutes with the right multiplication
  action of $\mathbb{Z}$ and thus of $A$. So $H$ admits the structure of a partially ordered real vector
  space, a dimension group, and an ordered $A$-module.

  We will assume that for $n>1$, the
  sum of each column of $M_{n}(1)$ are $1$. Let us introduce now the following notation:
  for $n > 1$, let $S_{n} = \mathbb{Z}\times \{1,2,\ldots,n\}$ denote the disjoint union of $n$ copies of $\mathbb{Z}$. The $k(n + 1)\times k(n)$-matrix $M_{n}$ defines a transition probability $P^{n}_{n+1}$ from $S_{k(n)}$ to $S_{k(n+1)}$ given by
  \[P^{n}_{n+1}((m,i),(p,j))=(\varphi^{n}_{j,i},x^{p-m})\text{ for }(m,i)\in S_{k(n)}\text{ and }(p,j)\in S_{k(n+1)}.\]
  By the group-invariant matrix-valued random walk on $\mathbb{Z}$ associated with the sequence $(M_{n})_{n\geq 1}$ is meant the Markov process consisting of the sequence of measurable spaces $(S_{k(n)})_{n\geq 1}$ with the transition probabilities $(P_{n+1}^{n})_{n\geq 1}$.

  Let $\psi=(\psi_{n})_{n\geq 1}$ be a bounded harmonic function, that is $\psi$ is an element of the von Neumann algebra $\oplus_{k\geq 1}l^{\infty}(S_{n(k)})$ satisfying
  \[\psi_{n}(x)=\sum_{y\in S_{n(k+1)}}P_{n+1}^{n}(x,y)\psi_{n}(y)\text{ for all }x\in S_{k(n)}.\]
  We assume that $\psi$ is $\mathbb{Z}$-invariant. Hence $\psi=(\psi_n)_{n\geq 1}$ is completely determined by a sequence of row vectors $\mu_{n}=(\mu_{n}^{1},\mu_{n}^{2},\ldots, \mu_{n}^{k(n)})\in \mathbb{R}^{k(n)}$ which are compatible with the maps $M_{n}(1)$ in the sense that $\mu_{n+1}M_{n}(1)=\mu_{n}$ and $\psi_n(k,i)=\mu_{n}^i$ for all $k\in\mathbb{Z}$. We denote by $\mathcal{H}_{\infty}^{+}$ the space of all bounded harmonic functions. To each $\psi\in \mathcal{H}_{\infty}^{+}$ it corresponds a bounded state $\gamma_\psi$ on $H$. This is a positive linear form on $(H,H^{+})$, defined  by
  $$\gamma_{\psi}([f,n])=\sum_{i=1}^{k(n)}\sum_{k\in \mathbb{Z}}(f_i,x^k)\mu_{n}^{i}, \text{ where } f=(f_1,f_2,\ldots, f_{k(n)})^{T}\in A^{k(n)}.$$
  Conversely, each bounded state of $H$ is of the form $\gamma_\psi$ for some bounded harmonic function $\psi\in\mathcal{H}_{\infty}^{+}$.  Moreover a bounded state $\gamma$ is $\mathbb{Z}$-invariant if and only if the corresponding bounded harmonic function $\psi$ is $\mathbb{Z}$-invariant.
  \noindent If  $f=(f_1, f_2, \ldots, f_{n(k)})^{T}$ is an element of $A^{k(n)}$ we put
  $$\|f\|_{\psi_n}=\sum_{i=1}^{k(n)}\sum_{k\in\mathbb{Z}}|(f_i,x^{k})|\psi_{n}(k,i).$$
  For $[f,n]\in H$ we define
  $$\|[f,n]\|_{\psi}=\liminf_{m\geq n}\|M_{m-1}M_{m-2}\cdots M_{n}f\|_{\psi_m}.$$
  In this paper we consider only the norm $\|\cdot\|$ corresponding to the state $\nu=(\nu_n)_{n\geq 1}$ where $\nu_n=(1,1,\ldots,1)\in \mathbb{R}^{k(n)}$.
  Let $E$ be the $\|\cdot\|$ completion of $H$ (i.e. the completion of $H/\{x\in H \|x\|_{\psi}=0\}$) with respect to this norm. Then there exist a standard measured $\mathbb{Z}$-space $(X,\mu)$ and an isometric order isomorphism $T : E\rightarrow L^{1}(X,\mu)$ which is also $\mathbb{Z}$-equivariant. The $\mathbb{Z}$-space $(X,\mu)$ is a realization of the Poisson boundary associated to the pair $H$  (see [GH]).
  Remark that the Poisson boundary of $H$ is isomorphic to the Poisson boundary of the group-invariant matrix-valued random walk on $\mathbb{Z}$ associated with the sequence $(M_{n})_{n\geq 1}$.

  The paper is organized as follows. In the second section we give a precise procedure by which to any nonsingular transformation $T$ that admits an adic realization, we associate a dimension space $H$. We will apply our procedure for some concrete examples.
  In the last section we constuct a countable family of finite measure preserving dynamical systems, to each of them we associate the corresponding dimension space and we prove that all these dimension spaces (and therefore the corresponding transformations) are AT. In order to prove this result we applied a criterion from \cite{GH}. Using the invariant $s$ introduced by Handelman in \cite{H2} we also obtain that these systems are mutually non isomorphic.

  \section{Dimension spaces associated to non-singular transformations}

  In order to associate a $\mathbb{Z}$-dimension space to a nonsingular automorphism of a measure space we will use the fact that the Mackey range of a cocycle defined on the tail equivalence relation of a Bratteli diagram  and taking values in $\mathbb{Z}$ can be identified with the Poisson boundary of a certain matrix valued random walk on $\mathbb{Z}$ (see [EG]).

  \begin{proposition}\label{mack}
  Let $T$ be a nonsingular automorphism of $(X, \mu)$ and denote by $\mathcal{R}$ the equivalence relation generated by $T$, that is $\mathcal{R} = \{(x, T ^{n}x), x \in X\}$. Let us consider the cocycle $c:\mathcal{R}\rightarrow\mathbb{Z}$, defined by
  $c(y, x) = n$, if $y = T^{n}x, n\in\mathbb{Z}$. Then $(X,\mu)$, which is a $\mathbb{Z}$-space together with the $\mathbb{Z}$-action  corresponding to $T$, is (a realization of) the Mackey range of $c$.
  \end{proposition}
  \begin{proof}
  Let $p$ be a probability measure on $\mathbb{Z}$ equivalent to the counting measure and denote by $\nu$ the product measure $\mu\times p$ on $X\times\mathbb{Z}$. Consider $\widetilde{\mathcal{R}}$, the equivalence relation on $(X \times\mathbb{Z},\nu)$ defined by $((x, m),(y, n)) \in \widetilde{\mathcal{R}}$ if $(x, y)\in \mathcal{R}$ and $n = m + c(x, y)$. Notice that $\widetilde{\mathcal{R}}$ is generated by the automorphism $\widetilde{T}$ of $(X\times\mathbb{Z},\nu)$, deined by \[\widetilde{T}(x, m) = (T x, m + c(x, T x)) = (T x, m-1)\text{ for }(x, m)\in X\times\mathbb{Z}.\]
  Let $\alpha$ be the action of $\mathbb{Z}$ on $L^{\infty}(X\times\mathbb{Z}, \mu\times p)$ given by $\alpha_{n}(f)(x, m) = f(x, m-n)$. Denote by $L^{\infty}(X\times\mathbb{Z}, \mu\times p)^{\widetilde{\mathcal{R}}}$ the space of all essentially $\widetilde{\mathcal{R}}$-invariant functions of  $L^{\infty}(X\times\mathbb{Z}, \mu\times p)$.
  Then, the Mackey range of $c$ is a point realization
  of the restriction of the action $\alpha$ to $L^{\infty}(X\times\mathbb{Z}, \mu\times p)^{\widetilde{\mathcal{R}}}$.
  Let
  \[\pi :L^{\infty}(X\times\mathbb{Z}, \mu\times p)^{\widetilde{\mathcal{R}}}\rightarrow L^{\infty}(X,\mu) \]
  defined by $\pi(f)(x) = f(x, 0)$, for $x\in X$. Since any essentially invariant function is
  almost everywhere equal to an invariant one we can assume that $f$ is $\widetilde{\mathcal{R}}$-invariant. Clearly,
  $f(x, n) = f(T^{n}x, 0) = \pi(f)(T^{n}x)$, for $n\in \mathbb{Z}$. It then follows that $\pi$ is an isomorphism. Let $\beta$ be the action of $\mathbb{Z}$ on $L^{\infty}(X,\mu)$ which corresponds to the action $\alpha$ via the isomorphism $\pi$. It is defined by $\beta_{n} = \pi\circ\alpha_{n}\circ\pi^{-1}$ for
  $n \in \mathbb{Z}$. It is a routine to check that $\beta_{n}(g)=g\circ T^{-n}$ for $g\in L^{\infty}(X,\mu)$ and $n\in\mathbb{Z}$, and then the proposition follows.
  \end{proof}

  \bigskip
  A Bratteli diagram $B=(V,E)$ consists of a vertex set $V$ and
  an edge set $E$, where $V$ and $E$ can be written as a countable
  disjoint union of nonempty finite sets, $V = V_{0}\cup V_{1}\cup
  V_{2}\cup \cdots$ and $E = E_{0}\cup E_{1}\cup\cdots$ with the
  following property: an edge $e$ in $E_{n}$ goes from a vertex in
  $V_{n}$ to one in $V_{n+1}$, which we denote by $s(e)$ and $r(e)$,
  respectively. We call $s$ the source map and $r$ the range map. We
  assume that $s^{-1}(v)$ is nonempty for each $v\in V$ and that
  $r^{-1}(v)$ is nonempty for each $v\in V\setminus V_{0}$. For
  simplicity, we assume that $V_{0}=\{v_{0}\}$ is a singleton.
  Let $B=(V,E)$ be a Bratteli diagram as above.
  For any vertex $v\notin V_{0}$ we totally order the set $r^{-1}(v)$ of edges reaching $v$
  (for example from left to the right). Then $B=(V,E)$ together with this order is called an ordered Bratteli diagram. Denote by $X_{n}$ the space
  of paths of length $n$
  \[X_{n}=\{(e_{0},e_{1},\ldots e_{n}), e_{k}\in E_{k},0\leqslant k\leqslant n \text{ and } r(e_{k})=s(e_{k+1})
  \text{ for }0\leqslant k\leqslant n-1\}.\] and $X$ the set of
  infinite paths
  \[X=\{\mathfrak{e}=(e_{0},e_{1},\ldots e_{k},\ldots ), e_{k}\in E_{k},\ r(e_{k})=s(e_{k+1})
  \text{ for }k \geq 0\}.\] On the paths space $X$ we consider the $\sigma$-algebra generated by sets of the form
  \[Z_{(f_{0},f_{1},\ldots f_{n})}=\{e\in X| \ e_{k}=f_{k}\text{ for } 0\leq k\leq
   n\}.\]These sets will be called cylinders.
  Let $\mu$ be a Markov measure (or AF measure) on $X$; it is determined by a
  system of transition probabilities $p$ (i.e. maps $p : E \rightarrow
  [0, 1]$ with $p(e)>0$ and $\sum_{\{e\in E, s(e)=v\}}p(e)=1$ for
  every $v\in V$) given by
  \[\mu(Z_{(f_{0},f_{1},\ldots f_{n})})=\prod_{k=0}^{n}p(f_{k})\]
  for each cylinder $Z_{(f_{0},f_{1},\ldots f_{n})}$.

  We define on $X$ a partial order by setting $\mathfrak{e}<\mathfrak{f}$ if there exists an integer $n\geq
  0$ such that $e_{n}<f_{n}$ and $e_{k}=f_{k}$, for $k>n$. Let
  $X_{\max}$ and $X_{\min}$ be the sets of the
  maximal and minimal paths, respectively. We define the adic (or Vershik) transformation
  $T :X\setminus X_{\max}\rightarrow X\setminus X_{\min}$ by setting $T\mathfrak{e}$ to be the smallest $\mathfrak{f}$ satisfying $\mathfrak{f}>\mathfrak{e}$. We assume that $X_{\max}$ and $X_{\min}$ have measure zero and so $T$ is nonsingular. Note that the equivalence relation on $X$, generated by  $T$ is, up to a null set, the tail equivalence relation $R$ on $X$.
  Given $b : E\rightarrow\mathbb{Z}$ we can define a cocycle $c:\mathcal{R}\rightarrow\mathbb{Z}$ by setting
  \[c(\mathfrak{f},\mathfrak{e})=\sum_{n\geq 0}(b(f_{n})-b(e_{n}))\]
  whenever $(\mathfrak{e},\mathfrak{f})\in\mathcal{R}$. We want to find a function $b : E\rightarrow\mathbb{Z}$ such
  that the corresponding cocycle $c$ satisfies $c(\mathfrak{f},\mathfrak{e}) = 1$ if $\mathfrak{f}$ is the successor of $\mathfrak{e}$ (equivalently
  $\mathfrak{f} = T(\mathfrak{e})$). We will define $b$ inductively, as follows.

  For every $e\in  E_{0}$, we put $b(e) = 1$. Having defined $b(e)$, for all $e \in E_0\cup E_1\cup E_{2}\cup\cdots \cup E_{n}$,
  we define $b(e)$ for $e\in E_{n+1}$ in the following way. For every vertex $v\in V_{n+1}$ let $p(v)$ be the number of edges reaching $v$.
  We write this set as $\{e^{1}<e^{2}<\cdots< e^{p(v)}\}$. Define
  \begin{align*}
  & b(e^{1})=0\\
  & b(e^{2})=\max\{b(e_{0})+b(e_{1})+\cdots +b(e_{n}) :  (e_{0},e_{1},\ldots, e_{n})\in X_{n}\text{ and }r(e_{n})=s(e^{1})\}+b(e^{1})+1\\
  &b(e^{3})=\max\{b(e_{0})+b(e_{1})+\cdots +b(e_{n}):  (e_{0},e_{1},\ldots, e_{n})\in X_{n}\text{ and }r(e_{n})=s(e^{2})\}+b(e^{2})+1\\
  &\cdots\cdots\cdots\cdots\cdots\cdots\cdots\cdots\cdots\cdots\cdots\cdots\cdots\cdots\cdots\cdots\cdots\cdots\cdots\cdots\cdots\cdots\cdots\cdots\cdots\cdots\cdots\\
  & b(e^{p(v)})=\max\{b(e_{0})+\cdots +b(e_{n}):  (e_{0},e_{1},\ldots, e_{n})\in X_{n}, \ r(e_{n})=s(e^{p(v)-1})\}+b(e^{p(v)-1})+1
  \end{align*}
    or equivalently,
  \begin{align*}
  & b(e^{1})=0\\
  & b(e^{2})=\max\{b(e_{0})+b(e_{1})+\cdots +b(e_{n})+ b(e^{1})+1:  (e_{0},e_{1},\ldots, e_{n},e^{1})\in X_{n+1}\}\\
  & b(e^{3})=\max\{b(e_{0})+b(e_{1})+\cdots +b(e_{n})+b(e^2)+1:  (e_{0},e_{1},\ldots, e_{n},e^{i})\in X_{n+1}\}\\
  &\cdots\cdots\cdots\cdots\cdots\cdots\cdots\cdots\cdots\cdots\cdots\cdots\cdots\cdots\cdots\cdots\cdots\cdots\cdots\cdots\cdots\cdots\cdots\cdots\cdots\cdots\cdots\\
  & b(e^{p(v)})=\max\{b(e_{0})+b(e_{1})+\cdots +b(e_{n})+b(e^{p(v)-1})+1:  (e_{0},e_{1},\ldots, e_{n},e^{p(v)})\in X_{n+1} \} \ \ \ \ \ \ \ \ \ \ \ \
  \end{align*}
  Then the cocycle corresponding to $b$ satisfies the conditions of Proposition \ref{mack}, and
  therefore, the space $(X,\mu)$ together with the $\mathbb{Z}$-action induced by the transformation $T$
  is a realization of the Mackey range of $c$.

  For $n\geqslant 0$, write $V_{n}=\{v_{i}^{n}, 1\leqslant i \leqslant k(n)\}$, where $k(n)=|V_{n}|$. Let us denote by $A_n=(a^n_{i,j})$ the incidence matrices corresponding to the above Bratteli diagram. Hence $a^n_{i,j}>0$ if there exists at least an edge connecting $v^{n}_{j}$ and $v^{n+1}_{i}$ and $a^n_{i,j}=0$ otherwise.

  Consider the matrix valued random walk defined by the sequence $(\sigma_{n})_{n\geqslant 1}$ of $k(n+1)\times k(n)$ matrices of measures, where $\sigma_{n,i,j}=0$ if $a^{n}_{ij}=0$ and $$\sigma_{n,i,j}=\sum_{s(e)=v^n_j, r(e)=v^{n+1}_i}p(e)\delta_{b(e)}$$ if $a^{n}_{i,j}\neq 0$ (by $\delta_{x}$ we denote the Dirac measure at $x$).
Then (see [EG], Remark 4.2) the Poisson boundary of the matrix valued random walk corresponding to $\{\sigma_n\}_{n\geq 1}$ is isomorphic, as a $\mathbb{Z}$-space, to the Mackey range of the cocycle $c$.

Then, following [GH] and Proposition \ref{mack} we have:

\begin{theorem}\label{reteta}
With the above notation, the dynamical system $(X,\mu,T)$ is isomorphic to the Poisson boundary of the dimension space corresponding to the sequence $(M_{n})$ of $k(n+1)\times k(n)$ matrices with entries  $M_{i,j}^n\in\mathbb{R}[x^{\pm 1}]$  given by $$M^n_{i,j}=\begin{cases}
\sum\limits_{s(e)=v^n_j, r(e)=v^{n+1}_i}p(e)x^{b(e)}&\text{ if } a^{n}_{i,j}\neq 0\\
 \ \ \ \ \ \ \ \ \ \ \ 0 &\text{ otherwise  }
\end{cases} $$
\end{theorem}

This procedure can be used in concrete situations, as we will explain bellow for the case of the dyadic odometer, the Morse automorphism and the irrational rotation of the circle.

 \subsection{ The dyadic odometer.} Let $B=(V,E)$ be the Bratteli diagram of the dyadic odometer. It has  one vertex $v_n$ at each level and there are exactly two edges $e^n_0$ and $e^n_1$ connecting $v^n$ and $v^{n+1}$. Hence $V=\cup_{n\geq 0}\{v_n\}$ is the set of vertices and $E=\cup_{n\geq 0} E_n$ with $E_n=\{e^n_0, e^n_1\}$ is the set of edges. We have a total order on the set of edges $E_n$ reaching the vertex $v_{n+1}$, given by $e^{n}_{0}<e^n_1$. Consider the measure $\mu$ corresponding to the system $p:E\rightarrow[0,1]$ of transition probabilities defined by $p(e)=\frac{1}{2}$.  As in the general case, we have a partial order on the path space $X$: for two paths $\mathfrak{e}$ and $\mathfrak{f}$ which are tail equivalent, we have $\mathfrak{e}<\mathfrak{f}$ if there exists an integer $n\geq 0$ such that $e_{n}<f_{n}$ and $e_{k}=f_{k}$, for $k>n$. Then the dyadic odometer is the adic transformation $T:X\rightarrow X$ defined by $T(\mathfrak{e})=\mathfrak{f}$, where $\mathfrak{f}$ is the successor of $\mathfrak{e}$.

We define the map $b:E\rightarrow\mathbb{Z}$. For the edges in $E_0$ define $$b(e^0_0)=0, \   b(e^0_1)=1.$$
For the edges in $E_1$ let
$$b(e^1_0)=0,\ b(e^1_0)=\max\{b(e^0_0), b(e^0_1)\}+1=2, $$ and for the edges of $E_2$, define
$$b(e^2_0)=0,\ b(e^2_0)=\max\{b(e^0_0)+b(e^1_0),b(e^0_0)+b(e^1_0),b(e^0_0)+b(e^1_0),b(e^0_0)+b(e^1_0)\}+1=4$$
In general, we have
$$b(e^n_0)=0, \ \ b(e^n_1)=2^{n-1}.$$
Then by Theorem \ref{reteta}, it results that the system $(X,\mu,T)$ is isomorphic as $\mathbb{Z}$-space to the Poisson boundary of the dimension space corresponding to $(P_n)$, where
$$P_n=\frac{1}{2}(1+x^{2^n}), \ \ n\geq 0.$$

\subsection{The Morse automorphism.} Let $V_0=\{v^0_0\}$ and $E_0=\{e^0_{0,0} e^0_{0,1}\}$. For $n\geq 1$, let $V_{n}=\{v^{n}_{0},v^{n}_{1}\}$ and $E_{n}=\{e^{n}_{0,0}, e^{n}_{0,1}, e^{n}_{1,0}, e^{n}_{1,1} \}$. Consider the Bratteli diagram $B=(V,E)$, where  $V=V_{0}\cup V_{1}\cup\cdots$ is the set of vertices, $E=E_{0}\cup E_{1}\cup\cdots$ is the set of edges and the edge $e^{n}_{i,j}$ connects the vertices $v^n_i$ and $v^{n+1}_j$.

  \begin{figure}
  \includegraphics[height=4in]{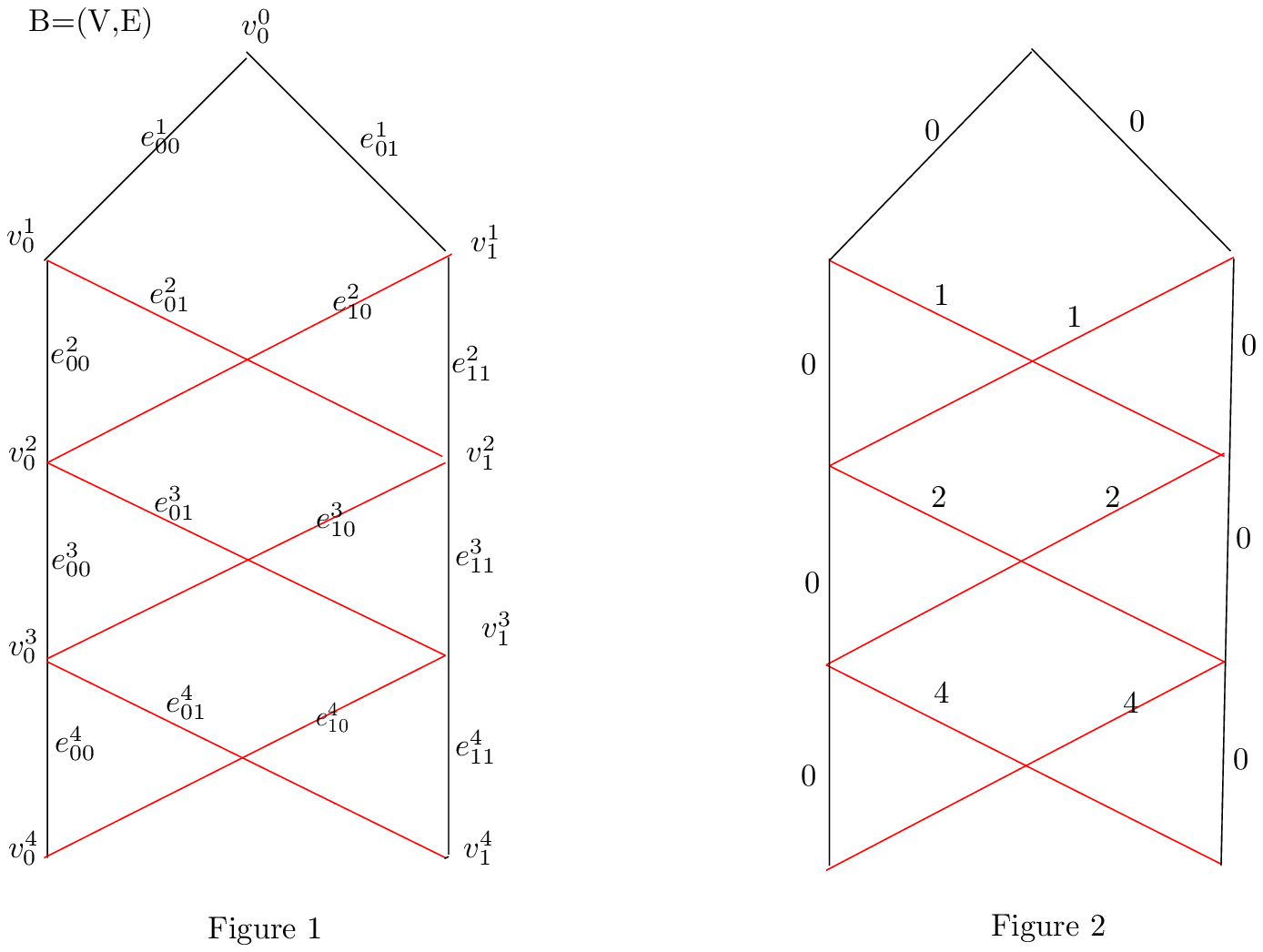}
  \end{figure}
  For $n\geqslant 1$, we order the set of edges reaching $v^{n+1}_{0}$ by $e^{n}_{0,0}<e^{n}_{1,0}$, and we order the set of edges reaching $v^{n+1}_{1}$ by $e^{n}_{1,1}<e^{n}_{0,1}$. We can then define a partial order on the path space $X$ as follows: if $\mathfrak{e}=(e_0,e_1,\ldots)$ and $\mathfrak{f}=(f_1,f_2,\ldots)$ are two paths which are tail equivalent we write $\mathfrak{e}<\mathfrak{f}$ if there exists an integer $n\geq 0$ such that $e_{n}<f_{n}$, $e_{k}=f_{k}$, for $k>n$.

  Denote by $\mu$ the Markov measure corresponding to the system  of transition probabilities $p:E\rightarrow [0,1]$, $p(e)=\frac{1}{2}$, for all $e\in E$. Then, the corresponding adic transformation (see Figure 1 and Figure 2) is (conjugate with) the Morse automorphism.

  We proceed now to define the function $b:E\rightarrow \mathbb{Z}$ as follows. For the edges in $E_{0}$ let
  \[b(e^{0}_{0,0})=b(e^{0}_{0,1})=0.\]
  For the edges in $E_{1}$ define
  \begin{align*}
  &b(e^{1}_{0,0})=b(e^{1}_{1,1})=0,\\
  &b(e^{1}_{1,0})=\max\{b(e^{1}_{0,0})\}+b(e^{1}_{0,0})+1=0+0+1=1\\
  &b(e^{1}_{0,1})=\max\{b(e^{1}_{0,1})\}+b(e^{1}_{1,1})+1=0+0+1=1.
  \end{align*}
  For the edges in $E_{2}$ we set
  \begin{align*}
 & b(e^{2}_{0,0})=b(e^{2}_{1,1})=0,\\
   & b(e^{2}_{1,0})=\max\{b(e^{0}_{0,0})+b(e^{1}_{0,0}),b(e^{0}_{0,1})+b(e^{1}_{1,0})\}+b(e^{2}_{0,0})+1=\max\{0,1\}+0+1=2\\    & b(e^{2}_{0,1})=\max\{b(e^{0}_{0,1})+b(e^{1}_{1,1}),b(e^{0}_{0,0})+b(e^{1}_{0,1})\}+b(e^{2}_{1,1})+1=\max\{0,1\}+0+1=2.
  \end{align*}
  Then following our general discussion, it follows that $(X,\mu)$ together with the $\mathbb{Z}$ action given by $X\mapsto T^nx$, $x\in X$ is isomorphic as $\mathbb{Z}$-space with the Poisson boundary of the $\mathbb{Z}$-dimension space corresponding to $(M_n)$, where  
     \[M_n=\frac{1}{2}\left[\begin{array}{cc}
      1 & x^{2^{n}} \\
      x^{2^{n}} & 1
       \end{array}
       \right], \ \   n\geqslant 0.\]

\subsection{Realization of the irrational rotation by cutting and stacking construction}

Even though any rank one transformation can be obtained by the cutting and stacking method there exist rank one automorphisms for which no such explicit realization is known. It is the case of the irrational rotation. We will show that if $\alpha$ is an irrational number whose continued fraction expansion
has certain properties then the corresponding rotation can be realized explicitly by the cutting and stacking method. Let $\alpha$ be an irrational number between $0$ and $1$ with continued fraction expansion
$$\alpha=\frac{1}{a(1)+\frac{1}{a(2)+\frac{1}{a(3)+\cdots}}}$$
Let $$\frac{p(n)}{q(n)}=\frac{1}{a(1)+\frac{1}{a(2)+\frac{1}{\cdots+\frac{1}{a(n)}}}}, \quad  n\geq 1$$ the successive approximations of $\alpha$. Let $p(0)=0$ and $q(0)=1$. Denote $$\alpha(n)=q(n)\left|\frac{p(n)}{q(n)}-\alpha\right|, \quad n\geq 0.$$
We construct an ordered Bratteli diagram $B=(V,E)$ as follows. Let $V=\cup_{n\geq 0}V_{n}$, with $V_{0}=\{v^{0}_{1}\}$ and $V_{n}=\{v^{n}_{1}, v^{n}_{2}\}$ be the vertex set. The set $E_{0}$ of edges connecting the set of vertices $V_{0}$ and $V_{1}$ consists of $a_{1}$ edges $\{e^0_{1,1}(k), i=1,\ldots a_{1}\}$ from $v^{0}_{1}$  to $v^{1}_{1}$ and one edge, denoted by $e^0_{1,2}$ from $v^{0}_{1}$ to $v^{1}_{2}$. In general the set of edges connecting $V_{n}$ and $V_{n+1}$ is $E_{n}=\{e^n_{2,1}, e^n_{1,2}, e^n_{1,1}(k), k=1,2,\ldots a(n+1) \}$ where the subscripts $i,j$ indicates that an edge connects the vertices  $v^{n}_{i}$ and $v^{n+1}_{j}$. Let $\mu$ pe the AF-measure defined by the the system $p:E\rightarrow[0,1]$, where
$$ p(e^0_{1,1}(k))=\alpha(0)=\alpha,\ k=1,2,\ldots, a(1), \ \  \ p(e^0_{1,2})=\alpha(1).$$
and for an arbitrary $n\geq 1$,
 $$p(e^n_{1,1}(k))=\frac{\alpha(n)}{\alpha(n-1)},\ k=1,2,\ldots, a(n+1) ; \quad p(e^n_{1,2})=\frac{\alpha(n+1)}{\alpha(n-1)}, \ \ p(e^n_{2,1})=1.$$
For every set of edges reaching a given vertex take the order from left to the right. Then, the adic transformation $T$ that corresponds to this order is measure theoretical isomorphic (conjugate) to $R_\alpha$, the irrational rotation of angle $\alpha$.
Define now $b:E_{n}\rightarrow\mathbb{Z}$ by
$$b(e^n_{1,2})=0, \
n\geq 0\ \ \ \ \ \ \  b(e^n_{2,1})=a(n+1)q(n), n\geq 1$$ and
$$b(e^n_{1,1}(k))=(k-1)q(n), \ \ \ \  k=1,2,\ldots a(n+1), \ n\geq 0$$
Then, by Theorem \ref{reteta}, it follows that the  $(X,\mu, T)$ is isomorphic, as a $\mathbb{Z}$-space, to the Poisson boundary of the dimension space corresponding to the sequence of matrices $(M_n)$, where
\[M_{0}=\left[\begin{array}{c}
\alpha(1+x+x^{2}+\cdots+x^{a(1)-1})\\\alpha(1) \end{array}\right] \]
and
\[M_{n}=\left[\begin{array}{cc}
  \frac{\alpha(n)}{\alpha(n-1)}\left(1+x^{q(n)}+x^{2q(n)}+\cdots +x^{(a(n+1)-1)q(n)}\right)& x^{a(n+1)q(n)} \\
   \frac{\alpha(n+1)}{\alpha(n-1)} & 0
     \end{array}
     \right], n\geqslant 1.\]
It is well known that $$\frac{1}{q(n)(q(n)+q(n+1))}<\left|\frac{p(n)}{q(n)}-\alpha\right|<\frac{1}{q(n)q(n+1)}\text { for all }n.$$
Therefore $$\frac{1}{q(n)+q(n+1)}<\alpha(n)<\frac{1}{q(n+1)},\text{ for all }n, $$
and then, for every $n$ we have
\begin{align*}
\frac{\alpha(n+1)}{\alpha(n-1)}&<\frac{q(n-1)+q(n+1)}{q(n+2)}<\frac{2q(n)}{q(n+2)}=\frac{2q(n)}{q(n+1)a(n+2)+q(n)}\\&<\frac{2q(n)}{q(n+1)a(n+2)}<\frac{1}{a(n+1)a(n+2)}.
\end{align*}
Let us assume that
\begin{equation}\label{sum}
\sum_{n=1}^{\infty}\frac{1}{a(n)a(n+1)}<\infty.
\end{equation}
We have then the following approximation of $M_n$ by a row-column product
\begin{equation} M_{n}\sim \left[\begin{array}{c}
1\\
0\end{array}\right]\cdot\left[\begin{array}{cc}
 \frac{1}{a(n+1)}\left(1+x^{q(n)}+x^{2q(n)}+\cdots +x^{(a(n+1)-1)q(n)}\right)& x^{a(n+1)q(n)}\end{array}\right]
\end{equation}
with an error of at most $\frac{2\alpha(n+1)}{\alpha(n-1)}<\frac{2}{a(n+1)a_{n+2}}$. Since by (\ref{sum}) the sums of all errors is finite we conclude:
\begin{proposition}\label{rank1}
If the terms of the continued fraction expansion of an irrational $\alpha$ satisfies (\ref{sum}), then the irrational rotation  $(\mathbb{T},\lambda, R_{\alpha})$ of angle $\alpha$ is isomorphic to the Poisson boundary of the dimension space corresponding  to the sequence of  polynomials $P_{n}$, $n\geq 0$, where
\begin{equation}\label{rot}
P_{n}=\frac{1}{a(n+1)}\left(1+x^{q(n)}+x^{2q(n)}+\cdots +x^{(a(n+1)-1)q(n)}\right).
\end{equation}
\end{proposition}
\bigskip

We will show that if the terms of the fraction expansion of $\alpha$ satisfy condition  (\ref{sum}), we have a concrete realization of the irrational rotation by the the cutting and stacking method. For simplicity, we assume that $a(n)\geq 2$, for all $n\geq 1$ (an explicit realization can also be achieved without this additional condition).

It is well known that $q(n+1)=a(n+1)q(n)+q(n-1)$, for $n\geq 2$, where  $q(0)=1$. Let $X=\prod_{n\geq 1}\{1,2,\ldots a(n)\}$, for $n\geq 1$. Denote by $T$ the product type odometer on $(X,\mu)$, where $\mu=\otimes_{n\geq 1}\mu_n$ with $\mu_{n}(i)=\frac{1}{k(n)}$ for $i\in X_n$ and $n\geq 1$.  Consider further, the function $h:X\rightarrow \mathbb{N}^*$ defined by
$$h(i,\ldots )=1,$$ for $i=1,2,\ldots a_1-1$, and $$h(a(1),a(2,\ldots a(n-1),i,\ldots)=q(n-1),$$ for $i=1,2,\ldots, a(n)-1$ and $n\geq 2$. Let us consider $Y=\{(x,n)\in X\times \mathbb{N}, 0\leq n <h(x)\}$ endowed with the measure $\mu$ which is the restriction to $Y$ of the product measure of $\nu$ with the counting measure on $\mathbb{Z}$. Denote by $T$ the usual product odometer corresponding to $X$. Define the transformation $S$ on $Y$ by
\begin{equation*}\label{transf}
S(x,n)=\begin{cases}
(x, n + 1)&\text{ if }n+1 < h(x)\\
(Tx,0) &\text{ if } n+1= h(x)
\end{cases}\end{equation*}
Let $E=\bigsqcup_{n\geq 1}X_{n}$ where $X_{n}=\{1,2,\ldots a(n)\}$. Define $b:E\rightarrow \mathbb{Z}$ by setting $b(k)=(k-1)q(n-1)$ for $k\in X_{n}$.
Let now $\mathcal{R}$ be the tail equivalence on $X$ and let $c:\mathcal{R}\rightarrow{Z}$ be the cocycle defined by  $$c(x,y)=\sum_{i=1}^{\infty} (b(x_{i})-b(y_{i})) \text{ for }(x,y)\in\mathcal{R}.$$
Then, Mackey range of $c$ is isomorphic to the Poisson boundary of the dimension space of $(P_n)$ from Proposition \ref{rank1}.
We have the following characterization of the Mackey range of $c$.
\begin{proposition}\label{r1}
The Mackey range of $c$ is isomorphic as a $\mathbb{Z}$-space to $(Y,\mu, S)$.
\end{proposition}
\noindent This is a consequence of the following more general result:
\begin{proposition}
Let $T$ be a nonsingular automorphism of $(X, \mu)$ and denote by $\mathcal{R}$ the equivalence relation generated by $T$, that is $\mathcal{R} = \{(x, T ^{n}x), x \in X\}$. Let $c:\mathcal{R}\rightarrow\mathbb{Z}$ be  a cocycle which satisfies $c(Tx, x)>0$ for every $x\in X$. Let $h: X\rightarrow \mathbb{Z}$ be a function defined by $h(x)=c(Tx,x)$ and let $Y=\{(x,n)\in X\times \mathbb{Z}\text{ with } 0\leq n< h(x)\}$ endowed with the measure $\nu$, the restriction to $Y$ of the product measure $\mu\times m$ on $X\times\mathbb{Z}$ where $m$ is the counting measure on $\mathbb{Z}$. Denote by $S$ the transformation on $Y$ defined by
$$
S(x,n)=\begin{cases}
(x, n + 1)&\text{ if }n+1 < h(x)\\
(Tx,0) &\text{ if } n+1= h(x).
\end{cases}
$$
Then the Mackey range of $c$ is isomorphic as a $\mathbb{Z}$-space to $(Y,\nu, S)$.
\end{proposition}
\begin{proof}
Consider $\widetilde{\mathcal{R}}$, the equivalence relation on $(X \times\mathbb{Z},\mu\times m)$ defined by $((x, m),(y, n)) \in \widetilde{\mathcal{R}}$ if $(x, y)\in \mathcal{R}$ and $n = m + c(x, y)$. In fact $\widetilde{\mathcal{R}}$ is generated by the automorphism $\widetilde{T}$ of $(X\times\mathbb{Z},\mu\times m)$ defined by \[\widetilde{T}(x, m) = (T x, m + c(x, T x)) = (T x, m-1)\text{ for }(x, m)\in X\times\mathbb{Z}.\]
Let $\alpha$ be the action of $\mathbb{Z}$ on $L^{\infty}(X\times\mathbb{Z}, \mu\times p)$ given by $\alpha_{n}(f)(x, m) = f(x, m-n)$. Denote by $L^{\infty}(X\times\mathbb{Z}, \mu\times p)^{\widetilde{\mathcal{R}}}$ the space of all essentially $\widetilde{\mathcal{R}}$-invariant functions of  $L^{\infty}(X\times\mathbb{Z}, \mu\times m)$.
Then, the Mackey range of $c$ is a point realization
of the restriction of the action $\alpha$ of $\mathbb{Z}$ to $L^{\infty}(X\times\mathbb{Z}, \mu\times m)^{\widetilde{\mathcal{R}}}$. Note that if $(x,m)\in X\times\mathbb{Z}$ and $c(T^{k}x,x)\leq m<c(T^{k+1}x,x)$ then $(T^{k}x, m+c(x,T^{k}x))$ is an element of $Y$ which is $\widetilde{\mathcal{R}}$ equivalent to $(x,m)$. Consequently any $\widetilde{\mathcal{R}}$ invariant function $f$ it is completely determined by the values it takes on $Y$. Then
\[\pi :L^{\infty}(X\times\mathbb{Z}, \mu\times m)^{\widetilde{\mathcal{R}}}\rightarrow L^{\infty}(Y,\nu), \text{  \ \ \ }\pi(f) = f|_{Y} \]is an isomorphism.
Let $\beta$ be the action of $\mathbb{Z}$ on $L^{\infty}(X,\mu)$
corresponding to $\alpha$ by the automorphism $\pi$. Hence for $n\geq 1$, we have $\beta_{n} =\pi\circ\alpha_{n}\circ\pi^{-1}$ and it is easy to check that $\beta_{n}(g)=g\circ S^{-n}$ for all $g\in L^{\infty}(Y,\nu)$,  $n\in\mathbb{Z}$. The proposition is proved.
\end{proof}
Aa an immediate consequence of Proposition \ref{r1} we have the following concrete realization of the irrational rotation by the cutting and stacking method.
\begin{remark}

 With the above notation, the irrational rotation of angle $\alpha$ can be obtain by the cutting and stacking method as follows:  we take the unit interval, divide it in $q(1)=a(1)$ equal intervals $I_{1,0}, I_{1,1}, \ldots I_{1, a(1)-1}$ and  stack them vertically in this order with $I_{1,0}$ at the bottom. Define a map $T_1$ on all these intervals except the last one, which maps linearly $I_{1,i}$ onto $I_{1,i+1}$. Having constructed a stack
$I_{n,0}, I_{n,1}, \ldots I_{n, a(n)q(n-1)-1}$ (which we view as a vertical column with $I_{n,0}$ at bottom) and a map $T_n$  defined on all these intervals except for the last one and sending linearly $I_{n,i}$ onto $I_{n,i+1}$, we divide this column into $a_{n+1}$ columns of the same width and place on the top of each of them $q(n-2)$ intervals of the same length, called spacers. More precisely, we divide each of $I_{n,i}$ in $a(n+1)$ intervals $I_{n,i}(k)$ $k=0,1,\ldots a(n+1)-1$ of equal length and for every $k=0,1,\ldots a(n+1)-1$ we place the spacers $I_{n,a(n)q(n-1)}(k), I_{n,a(n)q(n-1)+1}(k),\ldots, I_{n,a(n)q(n-1)+q(n-2)-2}(k), I_{n,q(n)-1}(k)$, of the same length as the intervals  $I_{n,i}(k)$, above the column $I_{n,0}(k), I_{n,1}(k),\ldots, I_{n,q(n)-1}(k)$. We stack all these intervals in a single column by putting $I_{n,0}(k+1)$ over $I_{n,q(n)-1}(k)$, for $k=0,1,\ldots a(n+1)-2$. We obtain a stack of $a(n+1)q(n)$ intervals which starting from the bottom will be denoted by $I_{n+1,0}, I_{n+1,1},\ldots, I_{n+1, a(n+1)q(n)-1}$. We have a map $T_{n+1}$, extending $T_n$ from the previous stage which  is defined everywhere except for $I_{n+1, a{n+1}q(n)-1}$ by mapping $I_{n+1,i}$ onto $I_{n+1,i}$. The space $X$ is the union of all intervals and $T(x)=\lim_{n\rightarrow\infty}T_n(x)$, for $x\in X$. Then $T$ is conjugate with $R_{\alpha}$, the rotation of angle $\alpha$.
\end{remark}

\section{A family dimension space which is AT}
In this section we construct a family $\{T_k, k=2,3,\ldots\}$ of (measure preserving) adic transformations which are approximately transitive (shortly AT) and mutually non isomorphic. We recall that AT property was introduced by Connes and Woods \cite{CW1} in the theory of von Neumann algebras, in order to characterize the Araki Woods factors among all AFD factors. The AT property was reformulated In terms of dimension spaces, by Giordano and Handelman \cite{GH}.

In order to define $T_k$ we construct a Bratteli diagram $B=(V,E)$ where the set of vertices
$V=\cup_{n\geq 0}V_n$, and the set of edges $E=\cup_{n\geq 0}E_n$ are  defined as follows:  $V_0=\{v^n_1\}$, $V_n=\{v^n_1, v^n_2,\ldots, v^n_k\}$  for $n\geq 1$, $E_0=\{e^0_1, e^0_2,\ldots ,e^0_k\}$ and $E_n=\{e^n_{i,i}, 1\leq i\leq k; \  e^n_{i,i+1}, 1\leq i\leq k-1; \ e^n_{k,1} \}$, for $n\geq 1$. An edge $e^n_{i,j}$ will connect the vertex $v^{n}_i$ and the vertex $v^{n+1}_j$.

For any $n\geq 2$, we order the set of edges reaching the vertex  $v^{n+1}_i\in V_n$ by $e^n_{i,i}<e^{n}_{i-1,i}$ if $2\leq i\leq k$ and $e^n_{1,1}<e^{n}_{k,1}$. We denote by $T_k$ the corresponding adic transformation.  We consider the measure $\mu$ determined by the system of transition probabilities $p:E\rightarrow[0,1]$, defined by $p(e)=\frac{1}{k}$ for $e\in E_0$ and $p(e)=\frac{1}{2}$ for $e\in E\setminus E_0$. Note that with respect to this measure $T_k$ is a nonsingular measure preserving transformation.
In Figure 3 and Figure 4 we have the ordered Bratteli diagrams corresponding to $n=3$ and $n=4$, where the order is black $<$ red.
The dimension space corresponding to $T_k$ is given by a sequence of circulant matrices $M_n$ of dimension $k\geq 2$ with $1$ on diagonal, $x^{2^{n}}$ bellow the diagonal $x^{2^{n}}$ in the upper right corner and $0$ otherwise, that is
  $M_n=\frac{1}{2}\left(I+x^{2^n}P\right)$, where $P$ is the standard cyclic permutation matrix.
  \begin{figure}
  \includegraphics[height=3in, width=6in]{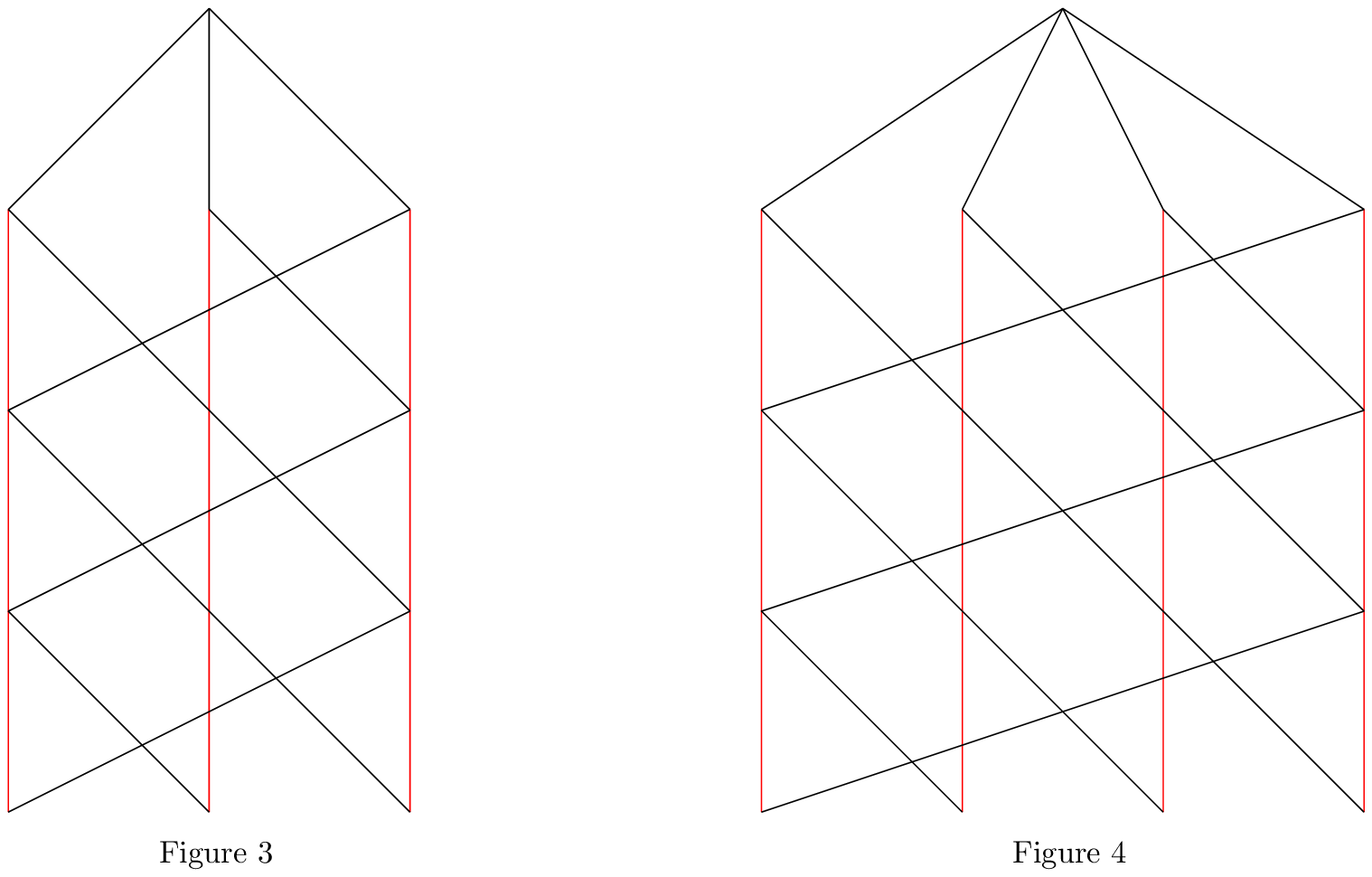}
  \end{figure}
  We will show that every $T_k$ is AT and that these transformation are pairwise nonisomorphic. The proof that the corresponding dimension space is AT generalizes the proof from \cite{DQ} of the fact that the Morse automorphism is AT, which occurs for $k=2$ and it follows the same idea for all $k$. For simplicity we will write the proof for $n=4$.
  \begin{proposition}
  The dimension space  $H$corresponding to $M_{n}=\frac12\left[
                                                   \begin{array}{cccc}
                                                     1 & 0 & 0 & x^{2^{n}} \\
                                                     x^{2^{n}} & 1 &0 & 0\\
                                                     0 & x^{2^{n}} & 1 & 0 \\
                                                     0 & 0 & x^{2^{n}} & 1 \\
                                                   \end{array}
                                                 \right]$, $n\geq 0$
  is AT.
  \end{proposition}
  \begin{proof}
  It will be enough to show for $M$ large enough, the matrix
  \[A=\prod_{j=0}^{N-1}\prod_{i=8Mj}^{8Mj+4M}\frac{1}{2}\left[\begin{array}{cccc}
                                                   1 & 0 & 0 &  x^{2^{i}}\\                                                                                                        x^{2^{i}} & 1 &0 & 0\\
                                                                                                        0 & x^{2^{i}} & 1 & 0 \\
                                                                                                        0 & 0 & x^{2^{i}} & 1 \\
                                                   \end{array}
                                                 \right]\]can be approximated arbitrarily well by a column row product.

  We write $p\sim q$ for $p=q \ (\text{mod }4)$. A monomial
  $X^{c_{8Mj}2^{8Mj}+c_{8Mj+1}2^{8Mj+1}\cdots+c_{8Mj+4M}2^{8Mj+4M}}$
  will be called basic monomial. For $p\in\{0,1,2,3\}$, let
  \[\varphi_{p}=\frac{1}{2^{N}}\sum_{\sum_{j=0}^{N-1}a_{j}\sim p
  }X^{\sum_{j=0}^{N-1}a_{j}2^{8Mj}}\] and
  \[f_{p}=\sum_{\sum_{j=0}^{N-1}\sum_{i=0}^{4M}c_{8Mj+i}\sim p}b_{c_{0}c_{9M}\ldots c_{8M(N-1)}}\prod_{j=0}^{N-1}X^{c_{8Mj}2^{8Mj}+c_{8Mj+1}2^{8Mj+1}\cdots+c_{8Mj+4M}2^{8Mj+4M}},\]where the basic monomials $X^{c_{8Mj}2^{8Mj}+c_{8Mj+1}2^{8Mj+1}\cdots+c_{8Mj+4M}2^{8Mj+4M}}$ appearing in $f_{p}$ are either of the form
  \begin{eqnarray}
  && X^{2^{8Mj}+2^{8Mj+1}\cdots+2^{8Mj+4M-1}} \text{ \ \ \ \ \  \ \ \ \  \ \ \ \ \ \ \  \ \ \ \ \  \ or }\label{a}\\
  && X^{c_{8Mj+1}2^{8Mj+1}+\cdots+c_{8Mj+4M}2^{8Mj+4M}}, \text { that is }c_{8Mj}=0 \label{b}
  \end{eqnarray}
  and
  \[b_{c_{0}c_{9M}\ldots c_{8M(N-1)}}=2^{N+2}2^{-3M\cdot\sum_{j=0}^{N-1}c_{8Mj}}(1-2^{-7M})^{N-\sum_{j=0}^{N-1}c_{8Mj}}.\]Notice that $\sum_{j=0}^{N-1}c_{8Mj}$ is the number of basic monomials of the form (\ref{a}) which appear in the product $\prod_{j=0}^{N-1}X^{c_{8Mj}2^{8Mj}+\cdots+c_{8Mj+4M}2^{8Mj+4M}}$.
  For $p\in\{0,1,2,3\}$ let \[g_{p}=\frac{f_{p}}{2^{(4M+1)N}}.\] For a
  monomial $P=X^{n}$, let us denote by $\#P$ the number of non-zero
  coefficients of the binary expansion of $n$. Notice that
  \[A=\left[
                    \begin{array}{cccc}
                      a_{0} & a_{3} & a_{2} & a_{1} \\
                      a_{1} & a_{0} & a_{3} & a_{2} \\
                      a_{2} & a_{1} & a_{0} & a_{3} \\
                      a_{3} & a_{2} & a_{1} & a_{0} \\
                    \end{array}
                  \right],
  \]
  where $a_{i}$ is a linear combination of monomials $P$ with $\#
  P\sim i$. We will show that for any $\varepsilon>0$, the product

  \[\left[
      \begin{array}{c}
        \varphi_{0} \\
        \varphi_{1} \\
        \varphi_{2} \\
        \varphi_{3} \\
      \end{array}
    \right]\cdot \left[
                   \begin{array}{cccc}
                     g_{0} & g_{3} & g_{2} & g_{1} \\
                   \end{array}
                 \right]
  \]
  approximates in norm $A$, up to $\varepsilon$, if $M$ is chosen large enough.

  It easily can be observed that if $P$ is a monomial of a
  $\varphi_{i}$ and $Q$ is a monomial of a $g_{j}$, then $PQ$ is a
  monomial of $a_{k}$ and $k\sim i+j$. Hence, it will be
  enough to show that
  \[\|\varphi_{i}(g_{1}+g_{2}+g_{3}+g_{4})-a_{1}-a_{2}-a_{3}-a_{4}\|\]
  can be made arbitrarily small for sufficiently large $M$. The numbers of terms appearing in
  $g_{1}+g_{2}+g_{3}+g_{4}$ which contain $S$ basic monomials of the
  form (\ref{a}) is
  \[\left(
                                                                                          \begin{array}{c}
                                                                                            N \\
                                                                                            S
                                                                                          \end{array}
                                                                                        \right)2^{4M(N-S)}
  .\]
  Then,
  \[\|g_{0}+g_{1}+g_{2}+g_{3}\|=\frac{1}{2^{N(4M+1)}}\sum \left(
                                                                                          \begin{array}{c}
                                                                                            N \\
                                                                                            S
                                                                                          \end{array}
                                                                                        \right)2^{4M(N-S)}2^{N+2}2^{-3MS}(1-2^{-7M})^{N-S}=2^{N+1}2^{4MN}=4.\]
  Notice that  for $i\in\{0,1,2,3\}$,  $$4\|\varphi_{i}\|=\|\varphi_{i}(g_{0}+g_{1}+g_{2}+g_{3})\|$$  and $\|\varphi_{i}\|$ are exponentially close to $\frac{1}{4}$  since
  \[\|\varphi_{i}- \frac{1}{4}\|\leq \frac{1}{2\cdot\sqrt{2}^{N(4M+1)}}.\]
We will compute the coefficients of the monomials of the matrix  \[\left[
      \begin{array}{c}
        \varphi_{0} \\
        \varphi_{1} \\
        \varphi_{2} \\
        \varphi_{3} \\
      \end{array}
    \right]\cdot \left[
                   \begin{array}{cccc}
                     f_{0} & f_{3} & f_{2} & f_{1} \\
                   \end{array}
                 \right]
  \]and we will show that most of them are uniformly close to $1$. Let $$P=\prod_{j=0}^{N-1}X^{c_{8Mj}2^{8Mj}+\cdots+c_{8Mj+4M}2^{8Mj+4M}}$$ and $p=\sum_{j=0}^{N-1}c_{8Mj}$. Let us assume that $P$ has $
  k_{0}$ basic monomials of the form (\ref{a}) and $k_{1}$ basic
  monomials of the form
  \begin{equation}\label{c}
  X^{2^{8Mj+4M}}.
  \end{equation}
  We want to see which are the monomials $A$ of one of the $\varphi_{p}$'s and which are the monomials
  $Q$ of one of the $f_{p}$'s such that $P=AQ$. Since $A$ is a monomials of a certain $\varphi_{p}$, $A=X^{\sum_{j=0}^{N-1}a_{j}2^{8Mj}}$.
  There is a unique choice for the basic monomials of $Q$ and for the $a_{j}$'s
  which correspond to the basic monomials of $P$ not of the form (\ref{a}) or (\ref{c}).
  Assume that $l_{0}$ of the $a_{j}$'s corresponding to basic monomials of the
  form (\ref{a}) are $0$ and $l_{1}$ of the $a_{j}$'s corresponding to
  basic monomials of the form (\ref{c}) are $1$. Then, $\# A=p-l_{0}+l_{1} \sim \#
  P-\# Q$. If $\# A \sim q$, then $p-l_{0}+l_{1}\sim q$. We compute
  the contribution to the coefficient of $P$ coming from $A$ with
  $l_{0}$ of the $a_{j}$'s corresponding to basic monomials of the form (\ref{a})
  equal to $0$ and $l_{1}$ of the $a_{j}$'s corresponding to basic monomials of
  the form (\ref{c}) equal to $1$. Then $Q$ must have exactly $l_{0}+l_{1}$
  basic monomials of the form (\ref{a}) and the contribution to the
  coefficient of $P$ is
  \[4\left(\begin{array}{c}
                                                 k_{0} \\
                                                 l_{0}
                                               \end{array}\right)\left(
                                                                   \begin{array}{c}
                                                                     k_{1} \\
                                                                     l_{1}\\
                                                                   \end{array}
                                                                 \right)2^{-3M}(l_{0}+l_{1})(1-2^{-7M})^{N-(l_{0}+l_{1})}
  \]
  Let $K=k_{0}+k_{1}$, summing over $l_{0}+l_{1}=L$ we  obtain
  \[4\left(
       \begin{array}{c}
         K \\
         L \\
       \end{array}
     \right)2^{-3ML}(1-2^{-7M})^{N-L}
  \]
  and summing over all $L=q$ (mod $4$), we get
  \[4\sum_{0\leq L\leq K; L\sim p }\left(
       \begin{array}{c}
         K \\
         L \\
       \end{array}
     \right)2^{-3ML}(1-2^{-7M})^{N-L}.
  \]
  For a monomial $P$, the quantity $K$ is the number of basic monomials of the form (\ref{a}) plus the number of basic monomials of form (\ref{c}) which  appear in $P$. Then by Cebyshev inequality it results that the number of monomials with $K$ between
  $2^{4M}-2^{8M/3}$ and $2^{4M}+2^{8M/3}$ is larger than $(1-2^{-4M/3})\cdot2^{(4M+1)N}$. For $K$ in this interval, we notice that for any $q=0,1,2,3$ the sums above are exponentially close to
  \[S(K)=\sum_{0\leq L\leq K}\left(
       \begin{array}{c}
         K \\
         L \\
       \end{array}
     \right)2^{-3ML}(1-2^{-7M})^{N-L}=(1-2^{-7M})^{N-K}(2^{-3M}+1-2^{-7M})^{K}.
  \]
  Notice that
  \[\frac{S(2^{4M}+z)}{S(2^{4M})}=\left(\frac{2^{-3M}+1-2^{-7M}}{1-2^{-7M}}\right)^{z},\]where $2^{-8M/3}<z<2^{8M/3}$. It follows that
  \[e^{-2^{-M/3}}<\frac{S(2^{4M}+z)}{S(2^{4M})}<e^{2^{-M/3}}.\]
  Moreover since $$S(2^{4M})\rightarrow 1,\text { as }M\rightarrow\infty$$ it follows that for $K$ in the above range, $S(K)$ differs from $S(2^{4M})$ by a constant which is exponentially close to $1$.

  We have
  \[a_{0}+a_{1}+a_{2}+a_{3}=\sum_{P}\frac{1}{2^{(4M+1)N}}P\text{ and }\varphi_{i}(g_{0}+g_{1}+g_{2}+g_{3})=\sum_{P}\frac{\alpha_{P}}{2^{(4M+1)N}}P.\]Let $\mathcal{C}$ be the set of all $P$ with the corresponding $K$ between $2^{4M}-2^{8M/3}$ and $2^{4M}+2^{8M/3}$. We have seen that the number of elements of $\mathcal{C}$ is at least $(1-2^{-4M/3})\cdot2^{(4M+1)N}$ and then  \[\sum_{P\notin\mathcal{C}}\frac{1}{2^{(4M+1)N}}<2^{-4M/3}.\]
  If $M$ is large enough it follows from the above discussion $|\alpha_{P}-1|<\varepsilon$, for all $P\in\mathcal{C}$. We have,
  \[\|\varphi_{i}(g_{0}+g_{1}+g_{2}+g_{3})-a_{0}+a_{1}+a_{2}+a_{3}\|\leq\sum_{P\in\mathcal{C}}\frac{1}{2^{(4M+1)N}}|\alpha_{P}-1|+
  \sum_{P\notin\mathcal{C}}\frac{1}{2^{(4M+1)N}}+\sum_{P\notin\mathcal{C}}\frac{\alpha_{P}}{2^{(4M+1)N}}.\]Then for $M$ large enough,  \[1-2\varepsilon<(1-\varepsilon)(1-2^{-4M/3})<\sum_{P\in\mathcal{C}}\frac{\alpha_{P}}{2^{(4M+1)N}}=\|\varphi_{i}(g_{0}+g_{1}+g_{2}+g_{3})\|-\sum_{P\notin\mathcal{C}}\frac{\alpha_{P}}{2^{(4M+1)N}}\]
  \[=\|\varphi_{i}\|-\sum_{P\notin\mathcal{C}}\frac{\alpha_{P}}{2^{(4M+1)N}}<1+\varepsilon-\sum_{P\notin\mathcal{C}}\frac{\alpha_{P}}{2^{(4M+1)N}},\]
  and therefore
  $$\sum_{P\notin\mathcal{C}}\frac{\alpha_{P}}{2^{(4M+1)N}}.$$
  Hence, for $M$ large enough,
  \[\|\varphi_{i}g_{0}-a_{0}\|+\|\varphi_{i}g_{1}-a_{1}\|+\|\varphi_{i}g_{2}-a_{2}\|+\|\varphi_{i}g_{3}-a_{3}\|=\|\varphi_{i}(g_{0}+g_{1}+g_{2}+g_{3})-a_{0}+a_{1}+a_{2}+a_{3}\|\]
  can be made arbitrarily small for all $i$. This is enough to conclude that the dimension space $H$ is AT.
  \end{proof}

\begin{theorem} With the above notation we have
\begin{itemize}
\item [1)] The nonsingular automorphisms $T_k$, $k\geq 2$ are AT.
\item [2)] All these automorphisms have the same eigenvalues as the dyadic odometer and they are not measure theoretically conjugate, as $k$ varies
\end{itemize}
\end{theorem}
\begin{proof}
The first part of the theorem results from the above proposition. The second part follows from \cite{H2}.
\end{proof}

\end{document}